\documentclass{article}

\usepackage{amsfonts}
\usepackage{amsmath}
\allowdisplaybreaks

\newtheorem{theorem}{Theorem}

\newtheorem{corollary}[theorem]{Corollary}

\newtheorem{lemma}[theorem]{Lemma}

\newtheorem{remark}[theorem]{Remark}

\newenvironment{proof}[1][Proof]{\noindent\textbf{#1.} }{\ \rule{0.5em}{0.5em}}

\begin{document}

\title {\bf Existence of mild solutions for a system of partial
differential equations with time-de\-pe\-ndent generators}

\author{{\bf Amanda del Carmen Andrade-Gonz\'{a}lez}\\
Universidad Aut\'{o}noma de Aguascalientes\\
Departamento de Matem\'{a}ticas y F\'{\i}sica\\
Aguascalientes, Aguascalientes, Mexico.\\
{\it acandra@correo.uaa.mx}\\ 
{\bf Jos\'{e} Villa-Morales}\\
Universidad Aut\'{o}noma de Aguascalientes\\
Departamento de Matem\'{a}ticas y F\'{\i}sica\\
Aguascalientes, Aguascalientes, Mexico.\\
{\it jvilla@correo.uaa.mx}}

\date{}

\maketitle

\begin{abstract}
We give sufficient conditions for global existence of positive mild
solutions for the weak coupled system:%
\begin{eqnarray*}
\frac{\partial u_{1}}{\partial t} &=&\rho _{1}t^{\rho _{1}-1}\Delta _{\alpha
_{1}}u_{1}+t^{\sigma _{1}}u_{2}^{\beta _{1}},\text{\ \ }u_{1}\left( 0\right)
=\varphi _{1}, \\
\frac{\partial u_{2}}{\partial t} &=&\rho _{2}t^{\rho _{2}-1}\Delta _{\alpha
_{2}}u_{2}+t^{\sigma _{2}}u_{1}^{\beta _{2}},\text{\ \ }u_{2}\left( 0\right)
=\varphi _{2},
\end{eqnarray*}%
where $\Delta _{\alpha _{i}}$ is a fractional Laplacian, $0<\alpha _{i}\leq
2,\ \beta _{i}>1,\ \rho _{i}>0,\ \sigma _{i}>-1\ $are constants and the
initial data $\varphi _{i}$ are positive, bounded and integrable functions.

\vspace{0.5cm}

\end{abstract}

\section{Introduction: statement of the results and ov\-er\-view}

Let $i\in \{1,2\}$ and $j=3-i$. In this paper we study the existence of
positive mild solutions of%
\begin{eqnarray}
\frac{\partial u_{i}\left( t,x\right) }{\partial t} &=&\rho _{i}t^{\rho
_{i}-1}\Delta _{\alpha _{i}}u_{i}\left( t,x\right) +t^{\sigma
_{i}}u_{j}^{\beta _{i}}\left( t,x\right) ,\text{\ \ }t>0,\text{ }x\in 
\mathbb{R}^{d},  \label{wcs} \\
u_{i}\left( 0,x\right) &=&\varphi _{i}\left( x\right) ,\text{\ \ }x\in 
\mathbb{R}^{d}.  \notag
\end{eqnarray}%
where $\Delta _{\alpha _{i}}:=-\left( -\Delta \right) ^{\alpha _{i}/2}$, $%
0<\alpha _{i}\leq 2$, is the $\alpha _{i}$-Laplacian, $\beta _{i}>1,$ $\rho
_{i}>0$, $\sigma _{i}>-1$ are constants and $\varphi _{i}$ are non negative,
not identically zero, bounded integrable functions.

The associated integral system of (\ref{wcs}) is%
\begin{equation}
u_{i}(t,x)=\int_{\mathbb{R}^{d}}p_{i}(t^{\rho _{i}},y-x)\varphi
_{i}(y)dy+\int_{0}^{t}\int_{\mathbb{R}^{d}}p_{i}(t^{\rho _{i}}-s^{\rho
_{i}},y-x)s^{\sigma _{i}}u_{j}^{\beta _{i}}(s,y)dyds.  \label{ecintegrl}
\end{equation}%
Here $p_{i}\left( t,x\right) $ denote the fundamental solution of $\frac{%
\partial }{\partial t}-\Delta _{\alpha _{i}}$ (in probability theory it is
called the symmetric $\alpha _{i}$-stable density). We say that $\left(
u_{1},u_{2}\right) $ is a mild solution of (\ref{wcs}) if $\left(
u_{1},u_{2}\right) $ is a solution of (\ref{ecintegrl}).

If there exist a solution $\left( u_{1},u_{2}\right) $ of (\ref{wcs})
defined in $\left[ 0,\infty \right) \times \mathbb{R}^{d}$, we say that $%
\left( u_{1},u_{2}\right) $ is a (classical) global solution, on the other
hand if there exists a number $t_{e}<\infty $ such that $\left(
u_{1},u_{2}\right) $ is unbounded in $\left[ 0,t\right] \times \mathbb{R}%
^{d} $, for each $t>t_{e}$, then we say that $\left( u_{1},u_{2}\right) $
blows up in finite time. It is well known that a classical solution is a
mild solution, but not vice versa. Therefore, if we give a sufficient
condition for global existence of positive solutions to (\ref{ecintegrl})
then we do not necessary have a condition for global existence of classical
solutions to (\ref{wcs}). Here we are going to deal with (mild) global
solutions.

Set $a\in \{1,2\}$ for which 
\begin{equation}
\alpha _{a}=\min \{\alpha _{1},\alpha _{2}\}\text{ \ and \ }b=3-a.
\label{defa}
\end{equation}

We also note that $i$ and $j$ are dummy variables, then if we define an
expression for $i$ we obtain other similar expression for $j$, changing only
the roles of the indices. For example, in the below inequality (\ref%
{condexiint}) is required $\tilde{x}_{j}$ and it is obtained from the
definition of $\tilde{x}_{i}$ given in (\ref{defxyro}). We are going to
follow this convention.

The main result is:

\begin{theorem}
\label{TeoPr} Addition to the above conditions on $\alpha _{i}$, $\beta _{i}$%
, $\rho _{i}$, $\sigma _{i}$ suppose that 
\begin{equation}
\max \left\{ \tilde{x_{i}},\tilde{x_{j}}\right\} <\min \left\{ 1,\tilde{\rho}%
_{i},\tilde{\rho}_{j},\max \{\tilde{k}_{i},\tilde{k}_{j}\}\right\} ,
\label{condexiint}
\end{equation}%
where 
\begin{equation}
\tilde{x}_{i}=\frac{1+\beta _{i}+\sigma _{i}(1-\beta _{i}\beta _{j})}{\beta
_{i}(1+\beta _{j})},\ \ \tilde{\rho}_{i}=\rho _{i}-\sigma _{i},
\label{defxyro}
\end{equation}%
and%
\begin{equation}
\tilde{k}_{i}=\frac{d\rho _{i}\rho _{j}(\beta _{i}\beta _{j}-1)-(\alpha
_{j}\rho _{i}\sigma _{j}+\alpha _{i}\beta _{j}\rho _{j}\sigma _{i})\beta _{i}%
}{\beta _{i}(\alpha _{j}\rho _{i}+\alpha _{i}\beta _{j}\rho _{j})}.
\label{defktilde}
\end{equation}%
If 
\begin{equation*}
\max \left\{ \tilde{x_{i}},\tilde{x_{j}}\right\} <\Delta <\min \left\{ 1,%
\tilde{\rho}_{i},\tilde{\rho}_{j},\max \{\tilde{k}_{i},\tilde{k}%
_{j}\}\right\}
\end{equation*}%
and $\varphi _{i}\in L_{+}^{\infty }(\mathbb{R}^{d})\cap L_{+}^{r_{i}}(%
\mathbb{R}^{d})$, where%
\begin{equation}
r_{i}=\frac{d\rho _{i}\rho _{j}(\beta _{i}\beta _{j}-1)}{\alpha _{i}\rho
_{j}(1+\beta _{i})+\alpha _{i}\rho _{j}\sigma _{i}+\beta _{i}\alpha _{j}\rho
_{i}\sigma _{j}+\beta _{i}(\alpha _{j}\rho _{i}-\alpha _{i}\rho _{j})\Delta }%
,  \label{defr2}
\end{equation}%
then (\ref{ecintegrl}) has a unique global solution $(u_{1},u_{2})$.
Moreover, there exists an $\varepsilon >0$ such that if $\left\Vert \varphi
_{i}\right\Vert _{r_{i}}+\left\Vert \varphi _{j}\right\Vert _{r_{j}}^{\beta
_{i}}<\varepsilon ,$ then%
\begin{equation}
t^{\xi _{i}}\left\Vert u_{i}(t)\right\Vert _{s_{i}}\leq c\varepsilon ,\ \
\forall t>0,  \label{estdslenls}
\end{equation}%
where $c$ is a positive constant%
\begin{equation}
s_{i}=\frac{d\rho _{i}\rho _{j}(\beta _{i}\beta _{j}-1)}{\alpha _{i}\rho
_{j}\sigma _{i}+\beta _{i}\alpha _{j}\rho _{i}\sigma _{j}+(\alpha _{i}\rho
_{j}+\beta _{i}\alpha _{j}\rho _{i})\Delta },  \label{defs2}
\end{equation}%
and 
\begin{equation*}
\xi _{i}=\frac{\alpha _{i}\rho _{j}-\Delta \alpha _{i}\rho _{j}+\alpha
_{i}\beta _{i}\rho _{j}-\Delta \alpha _{i}\beta _{i}\rho _{j}}{\alpha
_{i}\rho _{j}(\beta _{i}\beta _{j}-1)}.
\end{equation*}
\end{theorem}

Since we are dealing with an integral equation we just need to have the
solutions of (\ref{ecintegrl}) defined almost surely, this is what (\ref%
{estdslenls}) tell us. But imposing more restrictions we have that the
solutions of (\ref{ecintegrl}) are essentially bounded:

\begin{corollary}
\label{SegRe} Assume the hypothesis of Theorem \ref{TeoPr} and that 
\begin{equation}
\max \left\{ \tilde{x_{i}},\tilde{x_{j}}\right\} <\min \left\{ 1,\tilde{\rho}%
_{j},\tilde{\rho}_{i},\max \{\min \{\tilde{k}_{i},\hat{k}_{i}\},\min \{%
\tilde{k}_{j},\hat{k}_{j}\}\}\right\},  \label{cess}
\end{equation}%
where%
\begin{equation*}
\hat{k}_{i}=\frac{\alpha _{i}\rho _{i}\rho _{j}(\beta _{i}\beta
_{j}-1)-(\alpha _{j}\rho _{i}\sigma _{j}+\alpha _{i}\beta _{j}\rho
_{j}\sigma _{i})\beta _{i}}{\beta _{i}(\alpha _{j}\rho _{i}+\alpha _{i}\beta
_{j}\rho _{j})},
\end{equation*}%
then 
\begin{equation*}
||u_{i}(t)||_{\infty }\leq c||\varphi _{i}||_{\infty }+ct^{\sigma _{i}-\beta
_{i}\xi _{j}-\rho _{i}d\beta _{i}/(\alpha _{i}s_{j})+1},\ \ \forall t>0.
\end{equation*}
\end{corollary}

\begin{remark}
From the expression (\ref{expauxpkp}) of $\hat{k}_{i}$ we observe that $%
\alpha _{i}\geq d$ implies $\tilde{k}_{i}\leq \hat{k}_{i}$. Then, the bound (%
\ref{estdslenls}) for the solutions of (\ref{ecintegrl}) imply they are
essentially bounded. In others words, for small dimensions the solutions are
essentially bounded and integrable.
\end{remark}

Now, if the time-dependent generators are the same and we take a specific
initial data, then we get bounds for the solutions of (\ref{ecintegrl}). In
particular, this means that for these choice of parameters there are global
non trivial solutions for (\ref{ecintegrl}) with suitable initial conditions.

\begin{theorem}
\label{exiconicon} If 
\begin{equation*}
\alpha _{i}=\alpha _{j}=\alpha ,\ \ \rho _{i}=\rho _{j}=\rho \leq 1,
\end{equation*}%
and 
\begin{equation*}
\frac{1+\max \{\sigma _{i}+\beta _{i}(1+\sigma _{j}),\sigma _{j}+\beta
_{j}(1+\sigma _{i})\}}{\beta _{i}\beta _{j}-1}<\frac{d\rho }{\alpha },
\end{equation*}%
then there exists an $\varepsilon >0$ such that if 
\begin{equation*}
\varphi _{i}(x)=\varepsilon p(1,x),\ \ x\in \mathbb{R}^{d},
\end{equation*}%
then 
\begin{equation*}
u_{i}\left( t,x\right) \leq c\varepsilon (1+t)^{-k}(1+t^{\rho })^{d/\alpha
}p(1+t^{\rho },x),\ \ \forall (t,x)\in \lbrack 0,\infty )\times \mathbb{R}%
^{d},
\end{equation*}%
where $c$ and $k$ are positive constants.
\end{theorem}

In \cite{Fujita} Fujita shown (for the case $\alpha _{1}=\alpha _{2}=2,$ $%
\rho _{1}=\rho _{2}=1,$ $\sigma _{1}=\sigma _{2}=0$ and $\varphi
_{1}=\varphi _{2}$ in (\ref{wcs})) that $d=\alpha _{1}/\beta _{1}$ is the
critical dimension for blow up of (classical) solutions of (\ref{wcs}): if $%
d>\alpha _{1}/\beta _{1}$, then (\ref{wcs}) admits a global solution for all
sufficiently small initial conditions, whereas if $d<\alpha _{1}/\beta _{1}$%
, then for any non vanishing initial condition the solution blow up.

Since Fujita's pioneering work there are in the actuality a lot extensions.
For example, some works consider bounded domains, systems of equations,
others consider more general generators like elliptic operators, fractional
operators, etc (see \cite{Bai}, \cite{EH}, \cite{G-K-1}, \cite{moch}, \cite%
{Q-H}, \cite{Sug} and the references there in).

In this more general context some new phenomenon occurs. We mention some of
them:

\begin{itemize}
\item In general, the election of $r_{i}$ in $\left\Vert \varphi
_{i}\right\Vert _{r_{i}}$ depend of the choice of $\Delta $. But if, $\alpha
_{j}\rho _{i}=\alpha _{i}\rho _{j}$ then $r_{i}$ is independent of $\Delta $%
. As a particular case, if $\alpha _{1}=\alpha _{2}=2$ and $\rho _{1}=\rho
_{2}=1$ in (\ref{wcs}) then our result coincides with the Uda result (see
Theorem 4.2 in \cite{Uda}).

\item We observed that the estimations for global solutions depend on the
generator of $\Delta _{\alpha _{a}}$. On the other hand, the blow up
estimations depends on $\Delta _{\alpha _{b}}$ (see the results in \cite%
{villa}). From the interpretations of (\ref{wcs}) given in the introduction
of \cite{P-V} we could say that the blow up depends of the slow (diffusion)
motion of the particles and contrary the global existence of the fast motion
of the particles.

\item In the literature, the usual way of deal with the estimations required
for the solutions of (\ref{ecintegrl}) is throw the properties of the heat
equation, now we do not have such properties. To derive $L^{p}$ bounds for
the solutions we use a comparison result and the Banach fix point theorem
(see Lemmas \ref{lemacompa} and \ref{aeg}). A similar method was used in the
proof of Theorem \ref{exiconicon}.
\end{itemize}

In applied mathematics it is well known the importance of the study of
equations like (\ref{wcs}). In fact, for example, they arise in fields like
molecular biology, hydrodynamics and statistical physics \cite{S-Z}. Also,
notice that generators of the form $g_{i}\left( t\right) \Delta _{\alpha
_{i}}$ arise in models of anomalous growth of certain fractal interfaces 
\cite{M-W}.

The paper is organized as follows. In Section 2 we give some properties of
the symmetric $\alpha $-stable densities and provide some preliminary
results. In Section 3 we prove the main result, its corollary and Theorem %
\ref{exiconicon}.

\section{Preliminary results}

Let us start dealing with some properties of $p_{i}$.

\begin{lemma}
\label{pdd} Let $s,t>0$ and $x,y\in \mathbb{R}^{d}$, then

\begin{enumerate}
\item $p_{i}(ts,x)=t^{-d/\alpha _{i}}p_{i}(s,t^{-1/\alpha _{i}}x).$

\item $p_{i}(t,x)\geq \left( \frac{s}{t}\right) ^{d/\alpha _{i}}p_{i}(s,x)$%
,\ \ for $t\geq s.$
\end{enumerate}
\end{lemma}

\begin{proof}
See Section 2 in \cite{Sug}.\hfill
\end{proof}

\begin{lemma}
\label{cdd} There exists a constant $c\geq 1$ such that 
\begin{equation}
p_{i}(t,x)\leq cp_{a}(t^{\alpha _{a}/\alpha _{i}},x),\ \ \forall (t,x)\in
(0,\infty )\times \mathbb{R}^{d},  \label{eq:jvest1}
\end{equation}%
where $\alpha _{a}$ is defined in (\ref{defa}).
\end{lemma}

\begin{proof}
The inequality (\ref{eq:jvest1}) follows from Lemma 2.4 in \cite{M-V}.\hfill
\end{proof}

In what follows we will use $c$ to denote a positive and finite constant
whose value may vary from place to place.

For each bounded and measurable function $f:\mathbb{R}^{d}\rightarrow 
\mathbb{R}$ we have the semigroup property (in probability it is called the
Chapman-Kolmogorov equation):%
\begin{equation}
\int_{\mathbb{R}^{d}}\left( \int_{\mathbb{R}^{d}}f(z)p_{i}(t,y-z)dz\right)
p_{i}(s,x-y)dy=\int_{\mathbb{R}^{d}}f(y)p_{i}(t+s,x-y)dy.  \label{smgprop}
\end{equation}

\medskip

\begin{lemma}
\label{estmude} Let $\mu \geq 1$, then%
\begin{equation*}
\left\Vert p_{a}(t,\cdot )\right\Vert _{\mu }=ct^{-\frac{d}{\alpha _{a}}%
\left( 1-\frac{1}{\mu }\right) }.
\end{equation*}
\end{lemma}

\begin{proof}
By (1) in Lemma \ref{pdd} we get%
\begin{equation*}
\left\Vert p_{a}(t,\cdot )\right\Vert _{\mu }^{\mu }=t^{-d\mu /\alpha
_{a}}\int_{\mathbb{R}^{d}}p_{a}(1,t^{-1/\alpha _{a}}x)^{\mu }dx.
\end{equation*}%
The change of variable $z=t^{-1/\alpha _{a}}x$ implies%
\begin{equation*}
\left\Vert p_{a}(t,\cdot )\right\Vert _{\mu }^{\mu }=t^{\frac{d}{\alpha _{a}}%
(1-\mu )}\int_{\mathbb{R}^{d}}p_{a}\left( 1,z\right) ^{\mu }dz.
\end{equation*}%
From Theorem 2.1 in \cite{B-G} we have%
\begin{equation*}
\lim\limits_{\left\vert x\right\vert \rightarrow \infty }p_{a}(1,x)=0.
\end{equation*}%
Hence there exists $r>0$ such that%
\begin{equation*}
p_{a}(1,x)\leq 1,\ \ \forall \left\vert x\right\vert >r.
\end{equation*}%
Using this information we have
\begin{eqnarray*}
\int_{\mathbb{R}^{d}}p_{a}(1,x)^{\mu }dx &\leq &\int_{\left\vert
x\right\vert \leq r}p_{a}(1,x)^{\mu }dx+\int_{\left\vert x\right\vert
>r}p_{a}(1,x)dx \\
&\leq &\int_{\left\vert x\right\vert \leq r}p_{a}(1,x)^{\mu }dx+\int_{%
\mathbb{R}^{d}}p_{a}(1,x)dx \\
&\leq &\left\Vert p_{a}(1,\cdot )^{\mu }1_{\bar{B}(0,r)}(\cdot )\right\Vert
_{\infty }+1<\infty .
\end{eqnarray*}%
We used that $\left( p_{a}(1,\cdot )\right) ^{\mu }$ is a continuous
function on the compact set $\bar{B}(0,r)$ (the closed ball with center at
the origin and radius $r>0$) and $p_{a}(1,x)$ is a density.\hfill
\end{proof}

\medskip

We shall later require several times the following auxiliary tool.

\begin{lemma}
\label{lemacompa} Let $\varphi _{i}:\mathbb{R}^{d}\rightarrow \left[
0,\infty \right) $, $f:\left[ 0,\infty \right) \times \mathbb{R}^{d}\times 
\mathbb{R}^{d}\rightarrow \left[ 0,\infty \right) $ and $g:\left[ 0,\infty
\right) \times \left[ 0,\infty \right) \times \mathbb{R}^{d}\times \mathbb{R}%
^{d}\rightarrow \left[ 0,\infty \right) $ be continuous functions. Suppose
that for each $t\geq 0$ and $x\in \mathbb{R}^{d}$ the real-valued,
non-negative continuous functions $u_{i}$, $v_{i}$ satisfies 
\begin{equation*}
u_{i}(t,x)>\int_{\mathbb{R}^{d}}f(t,x,y)\varphi _{i}(y)dy+\int_{0}^{t}\int_{%
\mathbb{R}^{d}}g(t,s,x,y)u_{j}^{\beta _{i}}(s,y)dyds,
\end{equation*}%
and 
\begin{equation*}
v_{i}(t,x)\leq \int_{\mathbb{R}^{d}}f(t,x,y)\varphi
_{i}(y)dy+\int_{0}^{t}\int_{\mathbb{R}^{d}}g(t,s,x,y)v_{j}^{\beta
_{i}}(s,y)dyds.
\end{equation*}%
Then $u_{i}(t,x)\geq v_{i}(t,x)$, for each $(t,x)\in \left[ 0,\infty \right)
\times \mathbb{R}^{d}$.
\end{lemma}

\begin{proof}
Define 
\begin{equation*}
N_{i}=\left\{ t\geq 0:u_{i}(t,x)>v_{i}(t,x),\ \ \forall x\in \mathbb{R}%
^{d}\right\} .
\end{equation*}%
It is clear that $N_{i}\neq \oslash $ ($0\in N_{i}$). Let $t_{i}=\sup
N_{i}\in \lbrack 0,\infty ]$. We have the following cases.

\textbf{$t_{i}<\infty $ and $t_{j}<\infty $:} First observe that 
\begin{equation*}
u_{i}(t_{j},x)-v_{i}(t_{j},x)>\int_{0}^{t_{j}}\int_{\mathbb{R}%
^{d}}g(t,s,x,y)\{u_{j}^{\beta _{i}}(s,y)-v_{j}^{\beta _{i}}(s,y)\}dyds\geq 0,
\end{equation*}%
we used that the function $r\mapsto r^{\beta _{i}}$ is increasing. This
implies $t_{i}\geq t_{j}$. Analogously we deduce $t_{j}\geq t_{i}$.
Therefore $t_{i}=t_{j}$. The continuity of $(u_{i}-v_{i})(\cdot ,x)$ yields 
\begin{equation*}
0=u_{i}(t_{i},x)-v_{i}(t_{i},x)>\int_{0}^{t_{i}}\int_{\mathbb{R}%
^{d}}g(t,s,x,y)\{u_{j}^{\beta _{i}}(s,y)-v_{j}^{\beta _{i}}(s,y)\}dyds\geq 0.
\end{equation*}%
Which is a contradiction.

\textbf{$t_{i}=\infty $ and $t_{j}<\infty $ (or $t_{i}<\infty $ and $%
t_{j}=\infty $):} Here we have 
\begin{equation*}
0=u_{j}(t_{j},x)-v_{j}(t_{j},x)>\int_{0}^{t_{j}}\int_{\mathbb{R}%
^{d}}g(t,s,x,y)\{u_{i}^{\beta _{j}}(s,y)-v_{i}^{\beta _{j}}(s,y)\}dyds\geq 0.
\end{equation*}%
This also leads to a contradiction.

In this way, the only possibility is $t_{i}=\infty $ and $t_{j}=\infty $.
\hfill
\end{proof}

\begin{remark}
\label{remlema} Analogously, if for each $(t,x)\in \left[ 0,\infty \right)
\times \mathbb{R}^{d}$ the continuous functions $u_{i}$, $v_{i}$ satisfies 
\begin{equation*}
u_{i}(t,x)\geq\int_{\mathbb{R}^{d}}f(t,x,y)\varphi
_{i}(y)dy+\int_{0}^{t}\int_{\mathbb{R}^{d}}g(t,s,x,y)u_{j}^{\beta
_{i}}(s,y)dyds,
\end{equation*}%
and 
\begin{equation*}
v_{i}(t,x)< \int_{\mathbb{R}^{d}}f(t,x,y)\varphi _{i}(y)dy+\int_{0}^{t}\int_{%
\mathbb{R}^{d}}g(t,s,x,y)v_{j}^{\beta _{i}}(s,y)dyds,
\end{equation*}%
then $u_{i}(t,x)\geq v_{i}(t,x)$, for each $(t,x)\in \left[ 0,\infty \right)
\times \mathbb{R}^{d}$.
\end{remark}

\medskip

Let $s_{j}\geq 1$ and define 
\begin{equation*}
E_{\tau }=\left\{ u:[0,\tau ]\rightarrow L^{\infty }(\mathbb{R}^{d})\cap
L^{s_{j}}(\mathbb{R}^{d}),\ |||u|||<\infty \right\} ,
\end{equation*}%
where 
\begin{equation*}
|||u|||=\sup_{0\leq t\leq \tau }\left\{ ||u(t)||_{\infty
}+||u(t)||_{s_{j}}\right\} .
\end{equation*}%
Let $R>0$ and set 
\begin{equation*}
P_{\tau }=\left\{ u\in E_{\tau }:u\geq 0\right\} ,\ \ B_{R}=\left\{ u\in
E_{\tau }:|||u|||\leq R\right\} .
\end{equation*}%
Since $E_{\tau }$ is a Banach space and $P_{\tau }$, $B_{R}$ are closed
subspaces of $E_{\tau }$, then they are also Banch spaces. Let us also
define the functions $f_{j}:\{(t,s)\in \mathbb{R}^{2}:t\geq s\geq
0\}\rightarrow \mathbb{R}$ as%
\begin{equation*}
f_{j}(t,s)=(t^{\rho _{j}}-s^{\rho _{j}})^{\alpha _{a}/\alpha _{j}},
\end{equation*}%
and $f_{j}(t):=f_{j}(t,0)$, for $j\in \{1,2\}.$

\begin{lemma}
\label{aeg}If $\varphi _{j}\in L_{+}^{\infty }(\mathbb{R}^{d})\cap
L_{+}^{r_{j}}(\mathbb{R}^{d})$, $r_{j}\geq 1$, for $j\in \{1,2\}$ and $c>0$,
then the integral equation 
\begin{eqnarray*}
v_{j}(t,x) &=&c\int_{\mathbb{R}^{d}}p_{a}(f_{j}(t),y-x)\varphi _{j}(y)dy \\
&&+c\int_{0}^{t}\int_{\mathbb{R}^{d}}s^{\sigma
_{j}}p_{a}(f_{j}(t,s)+f_{i}(s),z-x)\varphi _{i}^{\beta _{j}}(z)dzds \\
&&+c\int_{0}^{t}\int_{\mathbb{R}^{d}}s^{\sigma _{j}+\beta
_{j}-1}\int_{0}^{s}p_{a}(f_{j}(t,s)+f_{i}(s,r),z-x) \\
&&\times r^{\beta _{j}\sigma _{i}}v_{j}^{\beta _{j}\beta _{i}}(r,z)drdzds
\end{eqnarray*}%
has a unique solution $v_{j}\in L^{\infty }([0,\tau ],L_{+}^{\infty }(%
\mathbb{R}^{d})\cap L_{+}^{s_{j}}(\mathbb{R}^{d}))$, for some $\tau >0$,
when 
\begin{equation}
s_{j}\geq r_{j}\text{ \ and \ }s_{j}\beta _{j}\geq r_{i}.  \label{prcond}
\end{equation}
\end{lemma}

\begin{proof}
Define the mapping $F:B_{R}\cap P_{\tau }\rightarrow L^{\infty }\left( [0,\tau
],L_{+}^{\infty }(\mathbb{R}^{d})\cap L_{+}^{s_{i}}(\mathbb{R}^{d})\right) $%
, as 
\begin{eqnarray}
F(\varphi )(t,x) &=&c\int_{\mathbb{R}^{d}}p_{a}(f_{j}(t),y-x)\varphi
_{j}(y)dy  \notag \\
&&+c\int_{0}^{t}\int_{\mathbb{R}^{d}}s^{\sigma
_{j}}p_{a}(f_{j}(t,s)+f_{i}(s),z-x)\varphi _{i}^{\beta _{j}}(z)dzds  \notag
\\
&&+c\int_{0}^{t}\int_{\mathbb{R}^{d}}s^{\sigma _{j}+\beta
_{j}-1}\int_{0}^{s}p_{a}(f_{j}(t,s)+f_{i}(s,r),z-x)  \notag \\
&&\times r^{\beta _{j}\sigma _{i}}\varphi ^{\beta _{j}\beta _{i}}(r,z)drdzds.
\label{estFelec}
\end{eqnarray}%
First we are going to see that $F$ is onto $B_{R}\cap P_{\tau }$. We take $(t,x)\in
\lbrack 0,\tau ]\times \mathbb{R}^{d}$ and see 
\begin{eqnarray*}
\left\vert F(\varphi )(t,x)\right\vert &\leq &c\left\Vert \varphi
_{j}\right\Vert _{\infty } \\
&&+c\left\Vert \varphi _{i}\right\Vert _{\infty }^{\beta
_{j}}\int_{0}^{t}s^{\sigma _{j}}ds+cR^{\beta _{j}\beta
_{i}}\int_{0}^{t}\int_{0}^{s}s^{\sigma _{j}+\beta _{j}-1}r^{\beta _{j}\sigma
_{i}}drds \\
&\leq &c\left\Vert \varphi _{j}\right\Vert _{\infty }+c\left\Vert \varphi
_{i}\right\Vert _{\infty }^{\beta _{j}}\tau ^{\sigma _{j}+1}+cR^{\beta
_{j}\beta _{i}}\tau ^{\sigma _{j}+1+\beta _{j}+\beta _{j}\sigma _{i}}.
\end{eqnarray*}%
Now let us deal with the $L^{s_{j}}(\mathbb{R}^{d})$ norm. By Jensen
inequality (see Theorem 14.16 in \cite{Y}) and using $r_{j}\leq s_{j}$ 
\begin{eqnarray*}
\left\Vert \int_{\mathbb{R}^{d}}p_{a}(f_{j}(t),y-\cdot )\varphi
_{j}(y)dy\right\Vert _{s_{j}} &\leq &\left( \int_{\mathbb{R}^{d}}\int_{%
\mathbb{R}^{d}}p_{a}(f_{j}(t),y-x)\varphi _{j}^{s_{j}}(y)dydx\right)
^{1/s_{j}} \\
&=&\left( \int_{\mathbb{R}^{d}}\varphi _{j}^{s_{j}}(y)dy\right) ^{1/s_{j}} \\
&\leq &\left( \int_{\mathbb{R}^{d}}\varphi _{j}(y)^{r_{j}}\left\Vert \varphi
_{j}\right\Vert _{\infty }^{s_{j}-r_{j}}dy\right) ^{1/s_{j}} \\
&=&\left\Vert \varphi _{j}\right\Vert _{\infty }^{1-r_{j}/s_{j}}\left\Vert
\varphi _{j}\right\Vert _{r_{j}}^{r_{j}/s_{j}}.
\end{eqnarray*}%
Analogously, since $s_{j}\beta _{j}\geq r_{i}$ we have by Minkowski integral inequality (see Theorem 23.69 in \cite{Y}) and Jensen inequality 
\begin{align*}
& \left\Vert \int_{0}^{t}\int_{\mathbb{R}^{d}}s^{\sigma
_{j}}p_{a}(f_{j}(t,s)+f_{i}(s),z-\cdot )\varphi _{i}^{\beta
_{j}}(z)dzds\right\Vert _{s_{j}} \\
& \leq \int_{0}^{t}s^{\sigma _{j}}\left\Vert \int_{\mathbb{R}%
^{d}}p_{a}(f_{j}(t,s)+f_{i}(s),z-x)\varphi _{i}^{\beta _{j}}(z)dz\right\Vert
_{s_{j}}ds \\
& \leq \int_{0}^{t}s^{\sigma _{j}}\left( \int_{\mathbb{R}^{d}}\int_{\mathbb{R%
}^{d}}p_{a}(f_{j}(t,s)+f_{i}(s),z-x)\varphi _{i}^{\beta
_{j}s_{j}}(z)dzdx\right) ^{1/s_{j}}ds \\
& =ct^{\sigma _{j}+1}\left( \int_{\mathbb{R}^{d}}\varphi _{i}^{\beta
_{j}s_{j}}(z)dz\right) ^{1/s_{j}}\leq c\tau ^{\sigma _{j}+1}\left\Vert
\varphi _{i}\right\Vert _{r_{i}}^{r_{i}/s_{j}}\left\Vert \varphi
_{i}\right\Vert _{\infty }^{\beta _{j}-r_{i}/s_{j}}.
\end{align*}%
Using that $\beta _{i}\beta _{j}\geq 1$ we get in the third summand in (\ref%
{estFelec}) 
\begin{align*}
& \left\Vert \int_{0}^{t}\int_{\mathbb{R}^{d}}s^{\sigma _{j}+\beta
_{j}-1}\int_{0}^{s}p_{a}(f_{j}(t,s)+f_{i}(s,r),z-\cdot )r^{\beta _{j}\sigma
_{i}}\varphi ^{\beta _{j}\beta _{i}}(r,z)drdzds\right\Vert _{s_{j}} \\
& \leq \int_{0}^{t}\int_{0}^{s}s^{\sigma _{j}+\beta _{j}-1}\left\Vert \int_{%
\mathbb{R}^{d}}p_{a}(f_{j}(t,s)+f_{i}(s,r),z-\cdot )\varphi ^{\beta
_{j}\beta _{i}}(r,z)dz\right\Vert _{s_{j}}r^{\beta _{j}\sigma _{i}}drds \\
& \leq \int_{0}^{t}\int_{0}^{s}s^{\sigma _{j}+\beta _{j}-1}r^{\beta
_{j}\sigma _{i}}\left( \int_{\mathbb{R}^{d}}\varphi ^{\beta _{j}\beta
_{i}s_{j}}(r,z)dz\right) ^{1/s_{j}}drds \\
& \leq \int_{0}^{t}\int_{0}^{s}s^{\sigma _{j}+\beta _{j}-1}r^{\beta
_{j}\sigma _{i}}\left\Vert \varphi (r,\cdot )\right\Vert _{\infty }^{\beta
_{j}\beta _{i}-1}\left\Vert \varphi (r,\cdot )\right\Vert _{s_{j}}drds \\
& \leq c\left( \sup_{r\leq t}\left\Vert \varphi (r,\cdot )\right\Vert
_{\infty }\right) ^{\beta _{j}\beta _{i}-1}\left( \sup_{r\leq t}\left\Vert
\varphi (r,\cdot )\right\Vert _{s_{j}}\right) t^{\sigma _{j}+\beta
_{j}+\beta _{j}\sigma _{i}+1} \\
& \leq cR^{\beta _{j}\beta i}\tau ^{\sigma _{j}+\beta _{j}+\beta _{j}\sigma
_{i}+1}.
\end{align*}%
If we take $R$ large enough such that 
\begin{equation*}
\frac{R}{2}\geq \left\Vert \varphi _{j}\right\Vert _{\infty
}^{1-r_{j}/s_{j}}\left\Vert \varphi _{j}\right\Vert
_{r_{j}}^{r_{j}/s_{j}}+c\left\Vert \varphi _{j}\right\Vert _{\infty },
\end{equation*}%
and $\tau $ small enough such that 
\begin{gather*}
c\left\Vert \varphi _{j}\right\Vert _{\infty }^{\beta _{j}}\tau ^{\sigma
_{j}+1}+cR^{\beta _{j}\beta _{i}}\tau ^{\sigma _{j}+\beta _{j}+\beta
_{j}\sigma _{i}+1}\ \ \ \ \ \ \ \ \ \ \ \ \ \ \ \ \ \ \ \ \ \ \ \ \ \ \ \ \
\ \ \ \ \  \\
+c\tau ^{\sigma _{j}+1}\left\Vert \varphi _{i}\right\Vert
_{r_{i}}^{r_{i}/s_{j}}\left\Vert \varphi _{j}\right\Vert _{\infty }^{\beta
_{j}-r_{i}/s_{j}}+cR^{\beta _{j}\beta _{i}}\tau ^{\sigma _{j}+\beta
_{j}+\beta _{j}\sigma _{i}+1}\leq \frac{R}{2},
\end{gather*}%
then for each $\varphi \in B_{R}\cap P_{\tau }$, 
\begin{equation*}
|||F(\varphi )|||=\sup_{t\leq \tau }\left\Vert F(\varphi )(t,\cdot
)\right\Vert _{\infty }+\sup_{t\leq \tau }\left\Vert F(\varphi )(t,\cdot
)\right\Vert _{s_{j}}\leq R.
\end{equation*}%
Now let us see that $F$ is a contraction. Let $\varphi ,\psi \in B_{R}\cap P_{\tau }$,
\begin{eqnarray*}
|F(\varphi )(t,x)-F(\psi )(t,x)| &\leq &c\int_{0}^{t}\int_{\mathbb{R}%
^{d}}s^{\sigma _{j}+\beta _{j}-1}\int_{0}^{s}p_{a}(f_{j}(t,s)+f_{i}(s,r),z-x)
\\
&&\times r^{\beta _{j}\sigma _{i}}\left\vert \varphi (r,z)^{\beta _{j}\beta
_{i}}-\psi (r,z)^{\beta _{j}\beta _{i}}\right\vert drdzds.
\end{eqnarray*}%
Using the elementary inequality 
\begin{equation*}
|s^{p}-r^{p}|\leq p(\max \{s,r\})^{p-1}|s-r|,\ \ s,r>0,\ p\geq 1,
\end{equation*}%
we have 
\begin{eqnarray*}
\left\vert F(\varphi )(t,x)-F(\psi )(t,x)\right\vert &\leq
&c\int_{0}^{t}s^{\sigma _{j}+\beta _{j}-1}\int_{0}^{s}r^{\beta _{j}\sigma
_{i}}\beta _{j}\beta _{i}R^{\beta _{j}\beta _{i}-1}drds \\
&&\times |||\varphi -\psi ||| \\
&\leq &c\tau ^{\sigma _{j}+\beta _{j}+\beta _{j}\sigma _{i}+1}|||\varphi
-\psi |||.
\end{eqnarray*}%
Also we choose $\tau >0$ small enough such that%
\begin{equation*}
c\tau ^{\sigma _{j}+\beta _{j}+\beta _{j}\sigma _{i}+1}<1.
\end{equation*}%
From this we have that $F$ is a contraction, then the result follows from
Banach fix point theorem.\hfill
\end{proof}

\begin{lemma}
\label{pretp}Suppose that $\varphi _{i}\in L_{+}^{\infty }(\mathbb{R}%
^{d})\cap L_{+}^{r_{i}}(\mathbb{R}^{d})$, $r_{i}\geq 1$, for $i\in \{1,2\}$.
Then there exits a local solution $(u_{1},u_{2})$ of (\ref{ecintegrl}).
Moreover, there exits $\tilde{T}>0$ such that $u_{i}\in L^{\infty }([0,%
\tilde{T}],L_{+}^{\infty }(\mathbb{R}^{d})\cap L_{+}^{s_{i}}(\mathbb{R}%
^{d})) $, for any $s_{i}$ satisfying (\ref{prcond}) and 
\begin{gather}
s_{i}\geq r_{i},\ \ s_{j}\geq \beta _{i}\text{, \ }s_{i}\beta _{i}\geq s_{j},
\label{consobiplem} \\
\frac{\beta _{i}}{s_{j}}-\frac{1}{s_{i}}<\frac{\alpha _{i}}{d}.
\label{conexleax}
\end{gather}
\end{lemma}

\begin{proof}
Proceeding as in Lemma \ref{aeg} we can find a real number $T>0$ such that $%
(u_{1},u_{2})$ is a solution of (\ref{ecintegrl}) in $[0,T]\times \mathbb{R}%
^{d}$ (see, for example, Theorem 3 in \cite{villa}). From (\ref{ecintegrl})
and (\ref{eq:jvest1})%
\begin{eqnarray}
u_{j}(t,x) &\leq &c\int_{\mathbb{R}^{d}}p_{a}(f_{j}(t),y-x)\varphi _{j}(y)dy
\notag \\
&&+c\int_{0}^{t}\int_{\mathbb{R}^{d}}p_{a}(f_{j}(t,s),y-x)s^{\sigma
_{j}}u_{i}^{\beta _{j}}(s,y)dyds.  \label{eq:eaeg}
\end{eqnarray}%
Using the elementary inequality%
\begin{equation*}
(s+r)^{q}\leq 2^{q-1}(s^{q}+r^{q}),\ \ q\geq 1,\ s,r\geq 0,
\end{equation*}%
in the previous estimation we have%
\begin{eqnarray}
u_{j}(t,x) &\leq &c\int_{\mathbb{R}^{d}}p_{a}(f_{j}(t),y-x)\varphi _{j}(y)dy
\notag \\
&&+c\int_{0}^{t}\int_{\mathbb{R}^{d}}p_{a}(f_{j}(t,s),y-x)s^{\sigma _{j}} 
\notag \\
&&\times \left( \int_{\mathbb{R}^{d}}cp_{a}(f_{i}(s),z-y)\varphi
_{i}(z)dz\right) ^{\beta _{j}}dyds  \notag \\
&&+c\int_{0}^{t}\int_{\mathbb{R}^{d}}p_{a}(f_{j}(t,s),y-x)s^{\sigma _{j}} 
\notag \\
&&\times \left( \int_{0}^{s}\int_{\mathbb{R}^{d}}cp_{a}(f_{i}(s,r),z-y)r^{%
\sigma _{i}}u_{j}^{\beta _{i}}(r,z)dzdr\right) ^{\beta _{j}}dyds.
\label{uaxtq}
\end{eqnarray}%
By Jensen inequality we obtain 
\begin{equation*}
\left( \int_{\mathbb{R}^{d}}cp_{a}(f_{i}(s),z-y)\varphi _{i}(z)dz\right)
^{\beta _{j}}\leq c^{\beta _{j}}\int_{\mathbb{R}^{d}}\varphi _{i}^{\beta
_{j}}(z)p_{a}(f_{i}(s),z-y)dz,
\end{equation*}%
and using again Jensen inequality%
\begin{align*}
& \left( \int_{0}^{s}\int_{\mathbb{R}^{d}}cp_{a}(f_{i}(s,r),z-y)r^{\sigma
_{i}}u_{j}^{\beta _{i}}(r,z)dzdr\right) ^{\beta _{j}} \\
& \leq c^{\beta _{j}}\left( \int_{0}^{s}\int_{\mathbb{R}%
^{d}}p_{a}(f_{i}(s,r),z-y)dzdr\right) ^{\beta _{j}-1} \\
& \times \int_{0}^{s}\int_{\mathbb{R}^{d}}p_{a}(f_{i}(s,r),z-y)r^{\beta
_{j}\sigma _{i}}u_{j}^{\beta _{j}\beta _{i}}(r,z)dzdr.
\end{align*}%
Taken into account this inequalities we deduce from (\ref{uaxtq}) 
\begin{eqnarray*}
u_{j}(t,x) &<&c\int_{\mathbb{R}^{d}}p_{a}(f_{j}(t),y-x)\varphi _{j}(y)dy
\\
&&+c\int_{0}^{t}\int_{\mathbb{R}^{d}}p_{a}\left( f_{j}(t,s),y-x\right)
s^{\sigma _{j}}\int_{\mathbb{R}^{d}}\varphi _{i}^{\beta
_{j}}(z)p_{a}(f_{i}(s),z-y)dzdyds \\
&&+c\int_{0}^{t}\int_{\mathbb{R}^{d}}p_{a}(f_{j}(t,s),y-x)s^{\sigma
_{j}+\beta _{j}-1} \\
&&\times \int_{0}^{s}\int_{\mathbb{R}^{d}}p_{a}\left( f_{i}(s,r),z-y\right)
r^{\beta _{j}\sigma _{i}}u_{j}^{\beta _{j}\beta _{i}}(r,z)dzdrdyds.
\end{eqnarray*}%
It is easy to see that the semigroup property (\ref{smgprop}) yields%
\begin{eqnarray*}
u_{j}(t,x) &<&c\int_{\mathbb{R}^{d}}p_{a}(f_{j}(t),y-x)\varphi _{j}(y)dy
\\
&&+c\int_{0}^{t}\int_{\mathbb{R}^{d}}s^{\sigma
_{j}}p_{a}(f_{j}(t,s)+f_{i}(s),z-x)\varphi _{i}^{\beta _{j}}(z)dzds \\
&&+c\int_{0}^{t}\int_{\mathbb{R}^{d}}s^{\sigma _{j}+\beta
_{j}-1}\int_{0}^{s}p_{a}(f_{j}(t,s)+f_{i}(s,r),z-x) \\
&&\times r^{\beta _{j}\sigma _{i}}u_{j}^{\beta _{j}\beta _{i}}(r,z)drdzds.
\end{eqnarray*}%
By comparison Lemma \ref{lemacompa} (see Remark \ref{remlema}) we have%
\begin{equation*}
u_{j}(t,x)\leq v_{j}(t,x),\ \ \forall (t,x)\in (0,\infty )\times \mathbb{R}%
^{d},
\end{equation*}%
and Lemma \ref{aeg} implies $u_{j}\in L^{\infty }([0,\tau ],L_{+}^{\infty }(%
\mathbb{R}^{d})\cap L_{+}^{s_{j}}(\mathbb{R}^{d}))$. From (\ref{ecintegrl})
and Min\-kows\-ki inequality (see Theorem 16.17 in \cite{Y}) one has
\begin{eqnarray}
\left\Vert u_{i}(t)\right\Vert _{s_{i}} &\leq &c\left\Vert \int_{\mathbb{R}%
^{d}}p_{a}\left( f_{i}(t),y-\cdot \right) \varphi _{i}(y)dy\right\Vert
_{s_{i}}  \notag \\
&&+c\left\Vert \int_{0}^{t}\int_{\mathbb{R}^{d}}p_{a}(f_{i}(t,s),y-\cdot
)s^{\sigma _{i}}u_{j}^{\beta _{i}}\left( s,y\right) dyds\right\Vert _{s_{i}}
\notag \\
&:=&cJ_{1}+cJ_{2}.  \label{auxtped}
\end{eqnarray}%
Let us estimate $J_{1}$ first. Since $s_{i}\geq r_{i}$, the Young inequality and Lemma \ref{estmude} implies%
\begin{equation*}
J_{1}\leq \left\Vert \varphi _{i}\right\Vert _{r_{i}}\left\Vert p_{a}\left(
f_{i}(t),\cdot \right) \right\Vert _{\left( 1+\frac{1}{s_{i}}-\frac{1}{r_{i}}%
\right) ^{-1}}=c\left\Vert \varphi _{i}\right\Vert _{r_{i}}t^{-\frac{d\rho
_{i}}{\alpha _{i}}\left( \frac{1}{r_{i}}-\frac{1}{s_{i}}\right) }.
\end{equation*}%
Now using that $s_{j}\geq \beta _{i}$, $s_{i}\beta _{i}\geq s_{j}$, the
Minkowski integral inequality, Young inequality and Lemma \ref%
{estmude} we get%
\begin{eqnarray*}
J_{2} &\leq &\int_{0}^{t}\left\Vert \int_{\mathbb{R}^{d}}p_{a}(f_{i}(t,s),y-%
\cdot )u_{j}^{\beta _{i}}(s,y)dy\right\Vert _{s_{i}}s^{\sigma _{i}}ds \\
&\leq &c\int_{0}^{t}\left( t^{\rho _{i}}-s^{\rho _{i}}\right) ^{-\frac{d}{%
\alpha _{i}}\left( \frac{\beta _{i}}{s_{j}}-\frac{1}{s_{i}}\right)
}\left\Vert u_{j}(s)\right\Vert _{s_{j}}^{\beta _{i}}s^{\sigma _{i}}ds.
\end{eqnarray*}%
A change of variable in the above integral allow us to write%
\begin{eqnarray*}
J_{2} &\leq &\frac{c}{\rho _{i}}t^{-\frac{d\rho _{i}}{\alpha _{i}}\left( 
\frac{\beta _{i}}{s_{j}}-\frac{1}{s_{i}}\right) +\sigma _{i}+1}\left(
\sup_{s\leq t}\left\Vert u_{j}(s)\right\Vert _{s_{j}}\right) ^{\beta _{i}} \\
&&\times \int_{0}^{1}\left( 1-s\right) ^{-\frac{d}{\alpha _{i}}\left( \frac{%
\beta _{i}}{s_{j}}-\frac{1}{s_{i}}\right) }s^{\frac{\sigma _{i}}{\rho _{i}}+%
\frac{1}{\rho _{i}}-1}ds.
\end{eqnarray*}%
Putting this together we have that condition (\ref{conexleax}) implies that $%
\left\Vert u_{i}(t)\right\Vert _{s_{i}}$ is bounded for $t<\min \{\tau ,T\}$%
.\hfill 
\end{proof}

\section{Proof of results}

\begin{proof}[Proof of Theorem \protect\ref{TeoPr}]
The first steep will be to see that it is possible to choice $r_{i},$ $s_{i},
$ $r_{j}$ and $s_{j}$ such that the conditions in Lemma \ref{pretp} are satisfied. From (\ref{auxtped}) we obtain%
\begin{equation}
\left\Vert u_{i}(t)\right\Vert _{s_{i}}\leq c\left\Vert \varphi
_{i}\right\Vert _{r_{i}}t^{-\xi _{i}}+c\int_{0}^{t}s^{\sigma _{i}}(t^{\rho
_{i}}-s^{\rho _{i}})^{-\delta _{i}}\left\Vert u_{j}(s)\right\Vert
_{s_{j}}^{\beta _{i}}ds,  \label{eq:epuex}
\end{equation}%
where 
\begin{equation}
\xi _{i}=\frac{d\rho _{i}}{\alpha _{i}}\left( \frac{1}{r_{i}}-\frac{1}{s_{i}}%
\right) \text{ \ and \ }\delta _{i}=\frac{d}{\alpha _{i}}\left( \frac{\beta
_{i}}{s_{j}}-\frac{1}{s_{i}}\right) .  \label{eq:defjide}
\end{equation}%
Iterating the inequality (\ref{eq:epuex})%
\begin{eqnarray*}
\left\Vert u_{i}(t)\right\Vert _{s_{i}} &\leq &c\left\Vert \varphi
_{i}\right\Vert _{r_{i}}t^{-\xi _{i}}+c\left\Vert \varphi _{j}\right\Vert
_{r_{j}}^{\beta _{i}}\int_{0}^{t}s^{\sigma _{i}-\beta _{i}\xi _{j}}(t^{\rho
_{i}}-s^{\rho _{i}})^{-\delta _{i}}ds \\
&&+c\int_{0}^{t}s^{\sigma _{i}}(t^{\rho _{i}}-s^{\rho _{i}})^{-\delta
_{i}}\left( \int_{0}^{s}r^{\sigma _{j}}(s^{\rho _{j}}-r^{\rho
_{j}})^{-\delta _{j}}\left\Vert u_{i}(r)\right\Vert _{s_{i}}^{\beta
_{j}}dr\right) ^{\beta _{i}}ds.
\end{eqnarray*}%
Let $w_{i}(t)=t^{\xi _{i}}\left\Vert u_{i}(t)\right\Vert _{s_{i}}$, then%
\begin{eqnarray}
w_{i}(t) &\leq &c\left\Vert \varphi _{i}\right\Vert _{r_{i}}+c\left\Vert
\varphi _{j}\right\Vert _{r_{j}}^{\beta _{i}}t^{\xi
_{i}}\int_{0}^{t}s^{\sigma _{i}-\beta _{i}\xi _{j}}(t^{\rho _{i}}-s^{\rho
_{i}})^{-\delta _{i}}ds  \notag \\
&&+ct^{\xi _{i}}\int_{0}^{t}s^{\sigma _{i}}(t^{\rho _{i}}-s^{\rho
_{i}})^{-\delta _{i}}\left( \int_{0}^{s}r^{\sigma _{j}-\beta _{j}\xi
_{i}}(s^{\rho _{j}}-r^{\rho _{j}})^{-\delta _{j}}dr\right) ^{\beta _{i}}ds 
\notag \\
&&\times \left( \sup_{r\leq t}w_{i}(r)\right) ^{\beta _{j}\beta _{i}}.
\label{spatercpa}
\end{eqnarray}%
Making some change of variables we obtain%
\begin{eqnarray}
\int_{0}^{t}s^{\sigma _{i}-\beta _{i}\xi _{j}}(t^{\rho _{i}}-s^{\rho
_{i}})^{-\delta _{i}}ds&=&\frac{1}{\rho _{i}}t^{\sigma _{i}-\beta _{i}\xi
_{j}-\delta _{i}\rho _{i}+1} \notag \\
&&\times \int_{0}^{1}s^{\frac{\sigma
_{i}-\beta _{i}\xi _{j}}{\rho _{i}}+\frac{1}{\rho _{i}}-1}(1-s)^{-\delta _{i}}ds.
\label{pitpb}
\end{eqnarray}%
This integral is convergent if 
\begin{equation}
\delta _{i}<1,\ \ \sigma _{i}-\beta _{i}\xi _{j}+1>0.  \label{eq:ecua5}
\end{equation}%
In the same way, the second integral in third term in the right hand side of
(\ref{spatercpa}) is finite if 
\begin{equation}
\delta _{j}<1,\ \ \sigma _{j}-\beta _{j}\xi _{i}+1>0.  \label{ecaua7}
\end{equation}%
Whence%
\begin{eqnarray*}
w_{i}(t) &\leq &c\left\Vert \varphi _{i}\right\Vert _{r_{i}}+c\left\Vert
\varphi _{j}\right\Vert _{r_{j}}^{\beta _{i}}t^{\xi _{i}+\sigma _{i}-\beta
_{i}\xi _{j}-\delta _{i}\rho _{i}+1} \\
&&+ct^{\xi _{i}}\int_{0}^{t}s^{\sigma _{i}+(\sigma _{j}-\beta _{j}\xi
_{i}-\delta _{j}\rho _{j}+1)\beta _{i}}(t^{\rho _{i}}-s^{\rho
_{i}})^{-\delta _{i}}ds\left( \sup_{r\leq t}w_{i}(r)\right) ^{\beta
_{j}\beta _{i}},
\end{eqnarray*}%
is well defined if 
\begin{equation}
\sigma _{i}+(\sigma _{j}-\beta _{j}\xi _{i}-\delta _{j}\rho _{j}+1)\beta
_{i}+1>0.  \label{eq:ecua6}
\end{equation}%
This can be written as%
\begin{equation*}
w_{i}(t)\leq c\left\Vert \varphi _{i}\right\Vert _{r_{i}}+c\left\Vert
\varphi _{j}\right\Vert _{r_{j}}^{\beta _{i}}t^{\eta _{i}}+ct^{\theta
_{i}}\left( \sup_{r\leq t}w_{i}(r)\right) ^{\beta _{j}\beta _{i}},
\end{equation*}%
where 
\begin{align}
\eta _{i}& =\xi _{i}+\sigma _{i}-\beta _{i}\xi _{j}-\delta _{i}\rho _{i}+1,
\label{eq:C1} \\
\theta _{i}& =\sigma _{i}+\left[ \sigma _{j}-\beta _{j}\xi _{i}-\delta
_{j}\rho _{j}+1\right] \beta _{i}-\delta _{i}\rho _{i}+\xi _{i}+1.
\label{eq:C2}
\end{align}%
Let us take%
\begin{equation*}
\xi _{i}=(1-\Delta )x_{i}\text{ \ and \ }\xi _{j}=(1-\Delta )x_{j},
\end{equation*}%
for a convenient choice of $\Delta >0$.

Assume that%
\begin{equation}
\rho _{i}\delta _{i}-\sigma _{i}=\Delta =\rho _{j}\delta _{j}-\sigma _{j}.
\label{eq:acde}
\end{equation}%
If%
\begin{equation}
\theta _{i}=0,  \label{C1}
\end{equation}%
then (\ref{eq:C2}) implies%
\begin{equation*}
1+\beta _{i}(1-\beta _{j})x_{i}+x_{i}=0,
\end{equation*}%
whence%
\begin{equation*}
x_{i}=\frac{1+\beta _{i}}{\beta _{i}\beta _{j}-1}.
\end{equation*}%
Also we want 
\begin{equation}
\eta _{i}=0,  \label{C2}
\end{equation}%
then (\ref{eq:C1}) implies%
\begin{equation*}
x_{j}=\frac{1+\beta _{j}}{\beta _{i}\beta _{j}-1}.
\end{equation*}%
The assumption (\ref{eq:acde}) and definitions in (\ref{eq:defjide}) yields%
\begin{align*}
\frac{\beta _{i}}{s_{j}}-\frac{1}{s_{i}}& =\frac{\alpha _{i}}{d\rho _{i}}%
(\Delta +\sigma _{i}), \\
\frac{\beta _{j}}{s_{i}}-\frac{1}{s_{j}}& =\frac{\alpha _{j}}{d\rho _{j}}%
(\Delta +\sigma _{j}).
\end{align*}%
Solving such linear system of equations we find $s_{i}$ and $s_{j}$ given in
(\ref{defs2}). On the other hand, using (\ref{eq:defjide}) we get $r_{i}$
and $r_{j}$ given in (\ref{defr2}). The conditions (\ref{eq:ecua5}), (\ref%
{ecaua7}) and (\ref{eq:ecua6}) are satisfies if%
\begin{equation*}
\max \left\{ \tilde{x_{i}},\tilde{x_{j}}\right\} <\Delta <\min \left\{ 1,%
\tilde{\rho}_{i},\tilde{\rho}_{j},\tilde{k}_{i}\right\} ,
\end{equation*}%
where $\tilde{x}_{i},$ $\tilde{\rho}_{i}$ and $\tilde{k}_{i}$ are defined in
(\ref{defxyro}) and (\ref{defktilde}). For this election of $r_{i},$ $s_{i},$
$r_{j}$ and $s_{j}$ the conditions (\ref{prcond}), (\ref{consobiplem}) and (%
\ref{conexleax}) are also satisfied, then by Lemma \ref{pretp} we have a
local solution $(u_{1},u_{2})$ on $[0,\tilde{T}]\times \mathbb{R}^{d}$. If
we change the r\^{o}les of $i$ and $j$ in the above procedure we have%
\begin{equation*}
\max \left\{ \tilde{x_{i}},\tilde{x_{j}}\right\} <\Delta <\min \left\{ 1,%
\tilde{\rho}_{j},\tilde{\rho}_{i},\tilde{k}_{j}\right\} .
\end{equation*}%
This implies that the correct condition on $\Delta $ is%
\begin{equation*}
\max \left\{ \tilde{x_{i}},\tilde{x_{j}}\right\} <\Delta <\min \left\{ 1,%
\tilde{\rho}_{i},\tilde{\rho}_{j},\max \{\tilde{k}_{i},\tilde{k}%
_{j}\}\right\} .
\end{equation*}%
As a second steep we are going to deal with the $L^{s_{i}}(\mathbb{R}^{d})$
boundedness of $u_{i}$. Conditions (\ref{C1}) and (\ref{C2}) implies%
\begin{equation}
z_{i}(t)\leq c\left( \left\Vert \varphi _{i}\right\Vert _{r_{i}}+\left\Vert
\varphi _{j}\right\Vert _{r_{j}}^{\beta _{i}}\right) +cz_{i}^{\beta
_{i}\beta _{j}}(t),  \label{eq:ecfl}
\end{equation}%
where%
\begin{equation*}
z_{i}(t)=\sup_{r\leq t}w_{i}(r).
\end{equation*}%
If we take $\varphi _{i},\varphi _{j}$ small enough such that 
\begin{equation}
\left\Vert \varphi _{i}\right\Vert _{r_{i}}+\left\Vert \varphi
_{j}\right\Vert _{r_{j}}^{\beta _{i}}<(2c)^{\frac{\beta _{i}\beta _{j}}{%
1-\beta _{i}\beta _{j}}},  \label{eq:csmall}
\end{equation}%
then 
\begin{equation}
z_{i}(t)\leq 2c(\left\Vert \varphi _{i}\right\Vert _{r_{i}}+\left\Vert
\varphi _{j}\right\Vert _{r_{j}}^{\beta _{i}}),\;\;\forall t\geq 0.
\label{eq:acps}
\end{equation}%
In fact, if (\ref{eq:acps}) were false, since $z_{i}$ is continuous, the
intermediate value theorem would imply that there exists $t_{0}>0$ such that 
\begin{equation*}
z_{i}(t_{0})=2c(\left\Vert \varphi _{i}\right\Vert _{r_{i}}+\left\Vert
\varphi _{j}\right\Vert _{r_{j}}^{\beta _{i}}),
\end{equation*}%
then of (\ref{eq:ecfl}) we could conclude 
\begin{equation*}
\left\Vert \varphi _{i}\right\Vert _{r_{i}}+\left\Vert \varphi
_{j}\right\Vert _{r_{j}}^{\beta _{i}}\geq (2c)^{\frac{\beta _{i}\beta _{j}}{%
1-\beta _{i}\beta _{j}}},
\end{equation*}%
this is a contradiction with (\ref{eq:csmall}). The second part of the
statement follows from the definitions of $z_{i}$ and $w_{i}$ together with (%
\ref{eq:acps}).\hfill 
\end{proof}

\bigskip

\begin{proof}[Proof of\ Corollary \protect\ref{SegRe}]
Let us consider (\ref{eq:eaeg}), then%
\begin{eqnarray}
||u_{i}(t)||_{\infty } &\leq &c\left\Vert \int_{\mathbb{R}%
^{d}}p_{a}(f_{i}(t),y-\cdot )\varphi _{i}(y)dy\right\Vert _{\infty }  \notag
\\
&&+c\left\Vert \int_{0}^{t}\int_{\mathbb{R}^{d}}p_{a}(f_{i}(t,s),y-\cdot
)s^{\sigma _{i}}u_{j}^{\beta _{i}}(s,y)dyds\right\Vert _{\infty }.
\label{eess}
\end{eqnarray}%
Since $p_{a}(f_{i}(t),\cdot )$ is a density 
\begin{equation*}
\left\Vert \int_{\mathbb{R}^{d}}p_{a}(f_{i}(t),y-\cdot )\varphi
_{i}(y)dy\right\Vert _{\infty }\leq ||\varphi _{i}||_{\infty }\left\Vert
\int_{\mathbb{R}^{d}}p_{a}(f_{i}(t),y-\cdot )dy\right\Vert _{\infty
}=||\varphi _{i}||_{\infty }.
\end{equation*}%
To estimate the second term in (\ref{eess}) we use H\"{o}lder inequality (see Theorem 16.14 in \cite{Y}) and
Lemma \ref{estmude} 
\begin{eqnarray*}
\int_{\mathbb{R}^{d}}p_{a}(f_{i}(t,s),y-\cdot )s^{\sigma _{i}}u_{j}^{\beta
_{i}}(s,y)dyds &=&||p_{a}(f_{i}(t,s),\cdot -x)u_{j}^{\beta _{i}}(s,\cdot
)||_{1} \\
&\leq &||p_{a}(f_{i}(t,s),\cdot -x)||_{\left( 1-\frac{\beta _{i}}{s_{j}}%
\right) ^{-1}} \\
&&\times ||u_{j}^{\beta _{i}}(s,\cdot )||_{\frac{s_{j}}{\beta _{i}}} \\
&=&c(t^{\rho _{i}}-s^{\rho _{i}})^{-\frac{d\beta _{i}}{\alpha _{i}s_{j}}%
}||u_{j}(s,\cdot )||_{s_{j}}^{\beta _{i}}.
\end{eqnarray*}%
From (\ref{eess}) one has%
\begin{equation*}
||u_{i}(t)||_{\infty }\leq c||\varphi _{i}||_{\infty }+c\int_{0}^{t}(t^{\rho
_{i}}-s^{\rho _{i}})^{-\frac{d\beta _{i}}{\alpha _{i}s_{j}}}||u_{j}(s,\cdot
)||_{s_{j}}^{\beta _{i}}s^{\sigma _{i}}ds,
\end{equation*}%
and (\ref{estdslenls}) implies%
\begin{eqnarray*}
||u_{i}(t)||_{\infty } &\leq &c||\varphi _{i}||_{\infty
}+c\int_{0}^{t}(t^{\rho _{i}}-s^{\rho _{i}})^{-\frac{d\beta _{i}}{\alpha
_{i}s_{j}}}s^{\sigma _{i}-\beta _{i}\xi _{j}}ds \\
&=&c||\varphi _{i}||_{\infty }+ct^{\sigma _{i}-\beta _{i}\xi _{j}-\frac{\rho
_{i}d\beta _{i}}{\alpha _{i}s_{j}}+1}\int_{0}^{1}(1-s)^{-\frac{d\beta _{i}}{%
\alpha _{i}s_{j}}}s^{\frac{1}{\rho _{i}}(\sigma _{i}-\beta _{i}\xi _{j})+%
\frac{1}{\rho _{i}}-1}ds.
\end{eqnarray*}%
Proceeding as in (\ref{pitpb}) we deduce that the above integral is fine if (%
\ref{eq:ecua5}) holds and%
\begin{equation*}
\frac{d\beta _{i}}{\alpha _{i}}<s_{j}.
\end{equation*}%
The definition (\ref{defs2}) of $s_{j}$\ impose the condition%
\begin{equation}
\Delta <\frac{(\frac{\alpha _{i}}{d})[d\rho _{i}\rho _{j}(\beta _{i}\beta
_{j}-1)]-(\alpha _{j}\rho _{i}\sigma _{j}+\alpha _{i}\beta _{j}\rho
_{j}\sigma _{i})\beta _{i}}{\beta _{i}(\alpha _{j}\rho _{i}+\alpha _{i}\beta
_{j}\rho _{j})}:=\hat{k}_{i}.  \label{expauxpkp}
\end{equation}%
Whence%
\begin{equation*}
\Delta <\min \left\{ 1,\tilde{\rho}_{j},\tilde{\rho}_{i},\tilde{k}_{i},\hat{k%
}_{i}\right\} .
\end{equation*}%
Changing the r\^{o}les of $i$ and $j$ we have the condition (\ref{cess}%
).\hfill
\end{proof}

\bigskip

\begin{proof}[Proof of Theorem \protect\ref{exiconicon}]
Here we consider the equation (\ref{eq:eaeg}) with%
\begin{equation*}
\alpha _{i}=\alpha _{j}=\alpha \ \ \text{and}\ \ \rho _{i}=\rho _{j}=\rho ,
\end{equation*}%
and we write $p\left( t,x\right) $ instead of $p_{a}\left( t,x\right) $.
Under this considerations we study the solution of
\begin{equation}
u_{i}(t,x)\leq \int_{\mathbb{R}^{d}}cp\left( t^{\rho },y-x\right) \varphi
_{i}(y)dy+\int_{0}^{t}\int_{\mathbb{R}^{d}}cp\left( t^{\rho }-s^{\rho
},y-x\right) s^{\sigma _{i}}u_{j}^{\beta _{i}}(s,y)dyds,  \label{ecpatepa}
\end{equation}%
with the  initial condition%
\begin{equation*}
\varphi _{i}(x)=\varepsilon p(1,x),\ \ x\in \mathbb{R}^{d}.
\end{equation*}%
Define the functions $g,h:(0,\infty )\rightarrow (0,\infty )$ as 
\begin{equation*}
g(t)=\left( e^{t}-1\right) ^{\rho },\ \ h(t)=\left[ g(t)+1\right] ^{1/\alpha
}.
\end{equation*}%
Through a change of variable the inequality (\ref{ecpatepa}) can be
transformed into%
\begin{align*}
u_{i}(g^{1/\rho }(t),h(t)x)& \leq c\varepsilon \int_{\mathbb{R}%
^{d}}p\left( g(t),y-h(t)x\right) p(1,y)dy \\
& +c\int_{0}^{e^{t}-1}\int_{\mathbb{R}^{d}}p\left( g(t)-s^{\rho
},y-h(t)x\right) s^{\sigma _{i}}u_{j}^{\beta _{i}}(s,y)dyds \\
& =c\varepsilon p\left( g(t)+1,h(t)x\right)  \\
& +c\int_{0}^{t}\int_{\mathbb{R}^{d}}p\left( g(t)-g(s),y-h(t)x\right)
g^{\sigma _{i}/\rho }(s) \\
& \times u_{j}^{\beta _{i}}(g^{1/\rho }(s),y)dye^{s}ds \\
& =c\varepsilon (h(t))^{-d}p(1,x) \\
& +c\int_{0}^{t}\int_{\mathbb{R}^{d}}p\left( g(t)-g(s),h(s)y-h(t)x\right)
g^{\sigma _{i}/\rho }(s) \\
& \times u_{j}^{\beta _{i}}(g^{1/\rho }(s),h(s)y)h^{d}(s)dye^{s}ds.
\end{align*}%
Setting%
\begin{equation*}
\bar{u}_{i}(t,x)=u_{i}(g^{1/\rho }(t),h(t)x),\ \ t\geq 0,\;x\in \mathbb{R}%
^{d},
\end{equation*}%
we have%
\begin{eqnarray*}
\bar{u}_{i}(t,x) &\leq &c\varepsilon (h(t))^{-d}p(1,x) \\
&&+c\int_{0}^{t}g^{\sigma _{i}/\rho }(s)e^{s}h^{d}(s) \\
&&\times \int_{\mathbb{R}^{d}}p\left( g(t)-g(s),h(s)y-h(t)x\right) \bar{u}%
_{j}^{\beta _{i}}(s,y)dyds.
\end{eqnarray*}%
Also define%
\begin{equation*}
w_{i}(t,x)=c\varepsilon e^{-(\theta _{i}+\eta _{i})t}p(1,x),\ \ t\geq
0,\;\;x\in \mathbb{R}^{d},
\end{equation*}%
where%
\begin{equation*}
\theta _{i}=\frac{1+\sigma _{i}+\beta _{i}(1+\sigma _{j})}{\beta _{i}\beta
_{j}-1},
\end{equation*}%
and $\eta _{i}$ is a positive number to be fixed. Observe that%
\begin{eqnarray*}
A_{i}(t,x) &=&\int_{0}^{t}g^{\sigma _{i}/\rho }(s)e^{s}h(t)^{d} \\
&&\times \int_{\mathbb{R}^{d}}p\left( g(t)-g(s),h(s)y-h(t)x\right)
w_{j}^{\beta _{i}}(s,y)dyds \\
&\leq &(c\varepsilon )^{\beta _{i}}\int_{0}^{t}g^{\sigma _{i}/\rho
}(s)e^{s}h^{d}(t) \\
&&\times \int_{\mathbb{R}^{d}}p\left( g(t)-g(s),h(s)y-h(t)x\right)  \\
&&\times e^{-(\theta _{j}+\eta _{j})s\beta _{i}}p(1,y)dyds\left( \sup_{z\in 
\mathbb{R}^{d}}p(1,z)\right) ^{\beta _{i}-1}.
\end{eqnarray*}%
Using property (1) in Lemma \ref{pdd} and the unimodality of $p(1,\cdot )$ ($%
p(1,x)\leq p(1,0)$, for each $x\in \mathbb{R}^{d}$) we get%
\begin{eqnarray*}
A_{i}(t,x) &<&(c\varepsilon )^{\beta _{i}}\int_{0}^{t}e^{s\sigma
_{i}+s-(\theta _{j}+\eta _{j})\beta _{i}s} \\
&&\times p\left( h(t)^{-\alpha }\left( g(t)-g(s)\right) +1,\frac{h(t)}{h(s)}%
x\right) dsp(1,x)^{\beta _{i}-1} \\
&=&(c\varepsilon )^{\beta _{i}}\left( p(1,0)\right) ^{\beta
_{i}-1}\int_{0}^{t}e^{s\sigma _{i}+s-(\theta _{j}+\eta _{j})\beta
_{i}s}dsp(1,x).
\end{eqnarray*}%
Since $\rho \leq 1$, then%
\begin{equation*}
\left( e^{s}-1\right) ^{\rho }+1\geq e^{\rho s},\ \ s\geq 0.
\end{equation*}%
To see this consider the function $\tilde{f}:\left[ 1,\infty \right)
\rightarrow \mathbb{R}$,%
\begin{equation*}
\tilde{f}(s)=(s-1)^{\rho }+1-s^{\rho },
\end{equation*}%
and observe that $\tilde{f}^{\prime }(s)\geq 0$. Taking into account this, 
\begin{align}
& c\varepsilon h(t)^{-d}p(1,x)+cA_{i}(t,x)  \notag \\
& \leq c\varepsilon e^{-d\rho t/\alpha }p(1,x)+(c\varepsilon )^{\beta
_{i}}p(1,0)^{\beta _{i}-1}p(1,x)\frac{e^{(\sigma _{i}+1-(\theta _{j}+\eta
_{j})\beta _{i})t}-1}{\sigma _{i}+1-(\theta _{j}+\eta _{j})\beta _{i}} \notag\\
&:=J_{3}+J_{4}.\label{estdt}
\end{align}%
Now let us estimate $J_{3}$ and $J_{4}$. For $J_{3}$ we have%
\begin{equation}
J_{3}=e^{(\theta _{i}+\eta _{i}-d\rho
/\alpha )t}w_{i}(t,x)<\frac{1}{2}w_{i}(t,x),  \label{est1a2}
\end{equation}%
because%
\begin{equation*}
\theta _{i}<\frac{d\rho }{\alpha },
\end{equation*}%
and taking $\eta _{i}$ small enough. And for $J_{4}$,
\begin{eqnarray*}
J_{4}&=&(c\varepsilon )^{\beta _{i}-1}p(1,0)^{\beta _{i}-1}\frac{e^{(\theta
_{i}+\eta _{i}+\sigma _{i}+1-(\theta _{j}+\eta _{j})\beta _{i})t}-e^{(\theta
_{i}+\eta _{i})t}}{\sigma _{i}+1-(\theta _{j}+\eta _{j})\beta _{i}}w_{i}(t,x) \\
& \leq& (c\varepsilon )^{\beta _{i}-1}p(1,0)^{\beta _{i}-1}\frac{e^{(\theta
_{i}+\eta _{i}+\sigma _{i}+1-(\theta _{j}+\eta _{j})\beta _{i})t}}{\sigma
_{i}+1-(\theta _{j}+\eta _{j})\beta _{i}}w_{i}(t,x).
\end{eqnarray*}
If we take%
\begin{equation*}
\eta _{j}>\frac{\eta _{i}}{\beta _{i}},
\end{equation*}%
then%
\begin{equation*}
\theta _{i}+\eta _{i}+\sigma _{i}+1-(\theta _{j}+\eta _{j})\beta _{i}<0,
\end{equation*}%
here we used that $\theta _{i}+\sigma _{i}+1=\theta _{j}\beta _{i}$. This
yields%
\begin{equation*}
J_{4}\leq (c\varepsilon )^{\beta _{i}-1}p(1,0)^{\beta
_{i}-1}w_{i}(t,x).  
\end{equation*}%
Therefore, if $\varepsilon >0$ is small enough then%
\begin{equation}
J_{4}\leq \frac{1}{2}w_{i}(t,x).\label{est2a2}
\end{equation}%
From (\ref{estdt}), (\ref{est1a2}) and (\ref{est2a2}) we get%
\begin{equation*}
c\varepsilon h(t)^{-d}p(1,x)+cA_{i}(t,x)<w_{i}(t,x).
\end{equation*}%
This means that%
\begin{eqnarray*}
w_{i}(t,x) &>&c\varepsilon h(t)^{-d}p(1,x)+c\int_{0}^{t}g^{\sigma _{i}/\rho
}(s)e^{s}h^{d}(s) \\
&&\times \int_{\mathbb{R}^{d}}p\left( g(t)-g(s),h(s)y-h(t)x\right)
w_{j}^{\beta _{i}}(s,y)dyds.
\end{eqnarray*}%
By the comparison Lemma \ref{lemacompa} we have%
\begin{equation*}
u_{i}\left( g^{1/\rho }(t),h(t)x\right) \leq c\varepsilon e^{-(\theta
_{i}+\eta _{i})t}p(1,x),\ \ \forall (t,x)\in \lbrack 0,\infty )\times 
\mathbb{R}^{d}.
\end{equation*}%
The results follows form (1) in Lemma \ref{pdd}.\hfill 
\end{proof}

\subsection*{Acknowledgment}

This work was partially supported by the grant No. 118294 of CONACyT.
Moreover, Villa-Morales was also supported by the grant PIM13-3N of UAA.

\end{document}